\documentclass{amsart}
\usepackage{amsmath}
\usepackage{amssymb}
\usepackage{mathrsfs}
  \usepackage{paralist}
  \usepackage{graphics} 
 \usepackage[colorlinks=true]{hyperref}

\newtheorem{theorem}{Theorem}

\theoremstyle{definition}

\newtheorem*{remark}{Remark}

\title[A Generalization of Gauss-Kuzmin-L\'evy Theorem]
      {A Generalization of Gauss-Kuzmin-L\'evy Theorem}

\author[Peng Sun]{}

\subjclass[2000]{Primary: 11J70, 11K50, 37C30.}
 \keywords{Gauss transformation, transfer operator, Gauss' problem, Hurwitz
 zeta function}

 \email{sunpeng@cufe.edu.cn}

\thanks{The author is supported by NSFC No. 11571387.}

\begin{document}

\maketitle\ 

\centerline{\scshape Peng Sun}
\medskip
{\footnotesize
 \centerline{China Economics and Management Academy}
   \centerline{Central University of Finance and Economics}
   \centerline{Beijing 100081, China}
} 

\bigskip

\begin{abstract}
    We prove a generalized Gauss-Kuzmin-L\'evy theorem
for the $p$-numerated generalized Gauss transformation $$T_p(x)=\{\frac
{p}{x}\}.$$
In addition, we give an estimate for the constant that appears in the theorem.

\end{abstract}



\bigskip

\bigskip


Let $p$ be a positive integer. We consider the following generalized Gauss transformation
on $[0,1]$
$$T(x)=T_p(x)=\begin{cases}0,& x=0,\\
\{\cfrac p x\},& x\ne 0,
\end{cases}$$
where $\{x\}$ is the fractional part of $x$.
Such transformations were first introduced in \cite{DK}
(the associated continued fractions had appeared in \cite{BG}) 
and also studied
in \cite{La1}\cite{La2}\cite{Sun1}.
For every $p$, $T_p$ has a unique absolutely continuous
ergodic invariant measure
$$d\mu_p(x)=\frac1{\ln (p+1)-\ln p}\cdot\frac 1{p+x}dm(x),$$
where $m$ is the Lebesuge measure on $[0,1]$. 
Equivalently, $$\eta_p(x)=\frac1{\ln (p+1)-\ln p}\cdot\cfrac{1}{p+x}$$
is the unique continuous eigenfunction of the transfer operator 
$$(\mathscr{G}_pf)(x)=\sum_{T_p(y)=x}\frac{f(y)}{|T_p'(y)|}=\sum_{k=p}^\infty\frac{p}{(k+x)^2}f(\frac{p}{k+x})$$
corresponding to the eigenvalue $1$. We remark that $\mathscr{G}_1$ is the so-called Gauss-Kuzmin-Wirsing
operator introduced in \cite{Khin}. Detailed discussion on
this operator can be found in \cite{IK}.

Denote
$$\varPhi_p(x)=\mu_p([0,x])=\frac{\ln(p+x)-\ln p}{\ln(p+1)-\ln p}.$$
Let
$$\varphi_n(x)=\varphi_{p,n}(x)=m(T_p^{-n}([0,x])).$$
Gauss has shown that 
$$\lim_{n\to\infty}\varphi_{1,n}(x)=\varPhi_1(x)=\frac{\ln(1+x)}{\ln2}.$$
In 1812, he proposed the problem to estimate
$$\Delta_{n}(x)=\varphi_{1,n}(x)-\varPhi_1(x).$$
The first solution was given by Kuzmin \cite{Kuzmin}, who showed in 1928 that
$$\Delta_{n}(x)=O(q^{\sqrt n})$$
as $n\to\infty$ for some $q\in(0,1)$. In 1929 L\'evy \cite{Levy} established
$$\Delta_{n}(x)=O(q^n)$$
for $q=3.5-2\sqrt2<0.7$.

In this paper we would like to follow an approach in \cite{RS} to generalize
L\'evy's result for all $T_p$:

\begin{theorem}\label{thmain}
For every positive integer $p$ and every $x\in[0,1]$,
\begin{equation}\label{Phiest}
\varphi_{p,n}(x)=\varPhi_p(x)+O(Q_p^n),
\end{equation}
where 
$$Q_p=2p^2\zeta(3,p)-p\zeta(2,p)<\frac1{2p}+\frac{3}{8p^2}<1.$$
and
$$\zeta(2,p)=\sum_{k=p}^\infty\frac{1}{k^2},\;
\zeta(3,p)=\sum_{k=p}^\infty\frac{1}{k^3}$$
are the Hurwitz zeta functions.
\end{theorem}

\begin{remark} 
We would like to thank an anonymous referee from whom we learned that there
is a similar result in \cite[Theorem 1.1 (ii)]{La1}. Compared to it, we have
a different approach and the
main novelty of Theorem \ref{thmain} is
the explicit expression of $Q_p$, which is an upper bound of the exponential rate
of decay for
$$\Delta_{p,n}=\varphi_{p,n}-\varPhi_p.$$
As a generalization of L\'evy's result \cite{Levy} on $\Delta_n=\Delta_{1,n}$, it is natural to expect
that $\Delta_{p,n}$ also decays exponentially. Our motivation is to see
how the rate depends on $p$.
The estimate we have for $Q_p$
shows that $Q_p\to 0$ as $p\to\infty$. Furthermore, it provides the first order term $\frac{1}{2p}$.
So generally it is faster that $\varphi_{p,n}$ converges to $\Phi_p$ as $p$ grows. This is the most interesting fact we
observe in
this work.
\end{remark}

For fixed $p$, we have $$\varphi_0(x)=\varphi_{p,0}(x)=x$$
and
\begin{align*}
\varphi_{n+1}(x)
=&m(T^{-(n+1)}((0,x)))=m(T^{-n}(T^{-1}((0,x))))\\
=&m(T^{-n}(\bigcup_{k=p}^\infty (\frac{p}{k+x},\frac{p}{k})))
=\sum_{k=p}^\infty(\varphi_n(\frac{p}{k})-\varphi_n(\frac{p}{k+x})).
\end{align*}
This recursive formula implies that $\varphi_n$ is differentiable (actually analytic)
and hence
$$\varphi_{n+1}'(x)=\sum_{k=p}^{\infty}\frac{p}{(k+x)^2}\varphi_n'(\frac{p}{k+x})=\mathscr{G}_p\varphi_n'.$$
So it is enough to study the operator $\mathscr{G}_p$. Note that (\ref{Phiest})
holds if
$$\varphi_n'(x)=\eta_p(x)+O(Q_p^n).$$
We can actually show a more
general
result:
\begin{theorem}\label{mainlemma}
Let $f\in C^1([0,1])$ such that $$\int_0^1f(x)dx=1.$$
Then for every positive integer $p$ and every $x\in[0,1]$,
$$(\mathscr{G}_p^nf)(x)=\eta_p(x)+O(Q_p^n).$$
\end{theorem}


\begin{proof}
Fix $p$. 
 Let 
 \begin{equation}\label{subgn}
 g_n(x)=(p+x)(\mathscr{G}_p^nf)(x).
 \end{equation}
Then
\begin{equation*}
g_{n+1}(x)=\sum_{k=p}^\infty g_n(\frac{p}{k+x})h_k(x),
\end{equation*}
where
$$h_k(x)=\frac{p+x}{(k+x)(k+1+x)}.$$
Note that
$$\sum_{k=p}^\infty h_k(x)=(p+x)\sum_{k=p}^\infty(\frac{1}{k+x}-\frac{1}{k+1+x})=1.$$
So
$$(\sum_{k=p}^\infty h_k(x))'=\sum_{k=p}^\infty h_k'(x)=0.$$
Moreover, for every $k\ge p$ and $x\in[0,1]$,
$$|h_k'(x)|=|\frac{(k+x)(k+1+x)-(2k+1+2x)(p+x)}{(k+x)^2(k+1+x)^2}|\le\frac{3}{k(k+1)}.$$
Note that $g_n\in C^1[0,1]$. Let
$$\|g_n\|_{C^1}=\max_{x\in[0,1]}|g_n(x)|+\max_{x\in[0,1]}|g_n'(x)|$$
be the norm of $g_n$ in $C^1[0,1]$. Then for every $k\ge p$ and $x\in[0,1]$,
\begin{align*}
|(g_n(\frac{p}{k+x})h_k(x))'|=
&|g_n(\frac{p}{k+x})h_k'(x)-\frac{p}{(k+x)^2}g_n'(\frac{p}{k+x})h_k(x)|
\\ \le&\|g_n\|_{C^1}(|h_k'(x)|+\frac{p}{(k+x)^2}|h_k(x)|)
\\ \le&\|g_n\|_{C^1}(\frac{3}{k(k+1)}+\frac{p(p+x)}{(k+x)^{3}(k+1+x)})
\\ \le&4\|g_n\|_{C^1}\cdot\frac{1}{k(k+1)}.
\end{align*}
So
$$\sum_{k=p}^\infty(g_n(\frac{p}{k+x})h_k(x))'=\sum_{k=p}^\infty (g_n(\frac{p}{k+x})h_k'(x)-\frac{p}{(k+x)^2}g_n'(\frac{p}{k+x})h_k(x))$$
converges absolutely and the sequence of its partial sums converges uniformly.
Hence we have
\begin{align*}
g_{n+1}'(x)=&\sum_{k=p}^\infty (g_n(\frac{p}{k+x})h_k'(x)-\frac{p}{(k+x)^2}g_n'(\frac{p}{k+x})h_k(x))\\
=&\sum_{k=p}^\infty(g_n(\frac{p}{k+x})-g_n(\frac{p}{p+x}))h_k'(x)-\sum_{k=p}^\infty
\frac{p}{(k+x)^2}g_n'(\frac{p}{k+x})h_k(x)\\
=&-\sum_{k=p}^\infty \frac{p(k-p)}{(p+x)(k+x)}g_n'(\frac{p}{\tau_k+x})h_k'(x)-\sum_{k=p}^\infty
\frac{p}{(k+x)^2}g_n'(\frac{p}{k+x})h_k(x)\\
\end{align*}
for some $\tau_k\in[p,k], k=p,p+1,\cdots$.

Let $\|\cdot\|$ be the maximum norm on $C[0,1]$, the space of all continuous
functions on $[0,1]$, and $M_n=\|g_n'\|$, i.e.
$$M_n=\max_{x\in[0,1]}|g_n'(x)|.$$
Then $$M_{n+1}\le M_n\|Q(x)\|$$
for
\begin{align*}Q(x)=\sum_{k=p}^\infty \frac{p(k-p)}{(p+x)(k+x)}h_k'(x)+\sum_{k=p}^\infty
\frac{p}{(k+x)^2}h_k(x)=p\sum_{k=p}^\infty D_k(x)
\end{align*}
and
\begin{align*}
D_k(x)=\frac{(p+1+x)(p+x)^2+(k-p)^2(k+1-p)}{(p+x)(k+x)^{3}(k+1+x)^2}\ge0\text{
for every }k\ge p.
\end{align*}

Let
\begin{align*}
G(k,x)=&(p+x)(k+x)^{3}(k+1+x)^2D_k'(x)\\
=&(p+x)^2(p+1+x)(\frac{2}{p+x}+\frac{1}{p+1+x})-\\
&((p+x)^2(p+1+x)+(k-p)^2(k+1-p))(\frac{1}{p+x}+\frac{3}{k+x}+\frac{2}{k+1+x}).
\end{align*}
If $p\le k\le 2p$, then
\begin{align*}
G(k,x)\le&(p+x)^2(p+1+x)(\frac{2}{p+x}+\frac{1}{p+1+x})-\\
&(p+x)^2(p+1+x)(\frac{1}{p+x}+\frac{3}{2p+x}+\frac{2}{2p+1+x})\\
<&0.
\end{align*}
If $k\ge2p+1$, then
\begin{align*}
G(k,x)\le&(p+x)^2(p+1+x)(\frac{1}{p+x}+\frac{1}{p+1+x})-\\
&((p+x)^2(p+1+x)+(k-p)^2(k+1-p))(\frac{3}{k+x})\\
:=&G_1(k,x).
\end{align*}
But for $k\ge 2p+1$,
\begin{align*}
-\frac{\partial G_1(k,x)}{\partial k}=&
\frac{3}{(k+x)^2}((k-p)^2(k+1-p)(\frac{2(k+x)}{k-p}+\frac{k+x}{k+1-p}-1)\\
&-(p+x)^2(p+1+x))\\
>&\frac{3}{(k+x)^2}((k-p)^2(2k+3x+p-1)-(p+x)^2(p+1+x))\\
>&0.
\end{align*}
So for $k\ge 2p+1$, 
\begin{align*}
G_1(k,x)\le&G_1(2p+1,x)\\
<&(p+x)^2(p+1+x)(\frac{1}{p+x}+\frac{1}{p+1+x})-\\
&((p+x)^2(p+1+x)+(p+1)^2(p+2))(\frac{3}{p+1+x})\\
<&(p+x)^2(p+1+x)(\frac{1}{p+x}+\frac{1}{p+x+1}-\frac{6}{p+1+x})\\
<&0.
\end{align*}
Therefore for all integers $k\ge p$ and all $x\in[0,1]$,
$G(k,x)<0.$
i.e. $D_k'(x)<0.$ Hence
\begin{align*}
&Q(x)=p\sum_{k=p}^\infty D_k(x)\le p\sum_{k=p}^\infty D_k(0)\\
=&\sum_{k=p}^\infty\frac{(p+1)p^2+(k-p)^2(k+1-p)}{k^{3}(k+1)^2}\\
=&\sum_{k=p}^\infty\frac{(k-p)^2(k+1)+p^2(2k+1)-pk^2}{k^3(k+1)^2}
\\
=&\sum_{k=p}^\infty\frac{1}{k(k+1)}-\sum_{k=p}^\infty\frac{2p}{k^{2}(k+1)}+\sum_{k=p}^\infty\frac{p^2}{k^{3}(k+1)}
+\sum_{k=p}^\infty\frac{p^2(2k+1)}{k^3(k+1)^2}-\sum_{k=p}^\infty\frac{p}{k(k+1)^2}\\
=&\sum_{k=p}^\infty(\frac1k-\frac1{k+1})-2p\sum_{k=p}^\infty(\frac{1}{k^2}-\frac{1}{k(k+1)})+p^2\sum_{k=p}^\infty(\frac{1}{k^{3}}-\frac{1}{k^2(k+1)})\\
&+p^2\sum_{k=p}^\infty(\frac{1}{k^3}-\frac{1}{k(k+1)^2})-p\sum_{k=p}^\infty(\frac{1}{k(k+1)}-\frac{1}{(k+1)^2})\\
=&\frac1p-2p(\zeta(2,p)-\frac1p)+p^2(\zeta(3,p)-\sum_{k=p}^\infty(\frac{1}{k^2}-\frac{1}{k(k+1)}))\\
&+p^2(\zeta(3,p)-\sum_{k=p}^\infty(\frac{1}{k(k+1)}-\frac{1}{(k+1)^2}))-p(\frac1p-(\zeta(2,p)+\frac{1}{p^2}))\\
=&-p\zeta(2,p)+1+2p^2\zeta(3,p)-p^2(\zeta(2,p)-\frac1p)-p^2(\frac1p-(\zeta(2,p)-\frac1{p^2}))\\
=&2p^2\zeta(3,p)-p\zeta(2,p)=Q_p.
\end{align*}

We will show in Theorem \ref{qpest} that $Q_p<1$. So $M_n=O(Q_p^n)$,
i.e. $g'_n(x)=O(Q_p^n)$. Then there is a  constant $c\in\mathbb{R}$ such that
\begin{equation}\label{gpconst}
(\mathscr{G}_p^nf)(x)=\frac{g_n(x)}{p+x}=\frac{c}{p+x}+O(Q_p^n).
\end{equation}
We note for any integrable function $\psi$,
$$\int_0^1(\mathscr{G}_p\psi)(x)dx=-\sum_{k=p}^\infty\int_0^1\psi(\frac{p}{k+x})d(\frac{p}{k+x})=\int_0^1\psi(x)dx.$$
So for all $n$, 
\begin{equation*}
\int_0^1(\mathscr{G}_p^nf)(x)dx=1.
\end{equation*}
Hence in (\ref{gpconst}) we must have 
$$c=\frac{1}{\ln(p+1)-\ln p}.$$
\end{proof}

\begin{remark}
A direct corollary of Theorem \ref{mainlemma} is
\begin{equation}\label{Glimit}
\lim_{n\to\infty}(\mathscr{G}_p^nf)(x)=\eta_p(x).
\end{equation}
However, the idea of the proof actually relies on the knowledge that $\eta_p(x)$
is invariant of $\mathscr{G}_p$: We presume that (\ref{Glimit}) holds. So we make the substitution (\ref{subgn}) and consider the derivative of $g_n$.
\end{remark}

Now we evaluate $Q_p$. If $p=1$ then $Q_1=2\zeta(3)-\zeta(2)<0.76$, where $\zeta(n)$ is the Riemann
zeta function. For $p\ge 2$ the following estimate is not too bad. 

\begin{theorem}\label{qpest}
For every positive integer $p$,
$$\frac1p-\frac1{2p+1}<Q_p<\frac1{2p}+\frac{3}{8p^2}<1.$$
\end{theorem}

\begin{remark}
This implies that
$$Q_p=\frac{1}{2p}+O(\frac{1}{p^2}).$$
Applying results on asymptotic expansions
of Hurwitz zeta functions or polygamma functions (cf. \cite{AbrS}) we actually
have
$$Q_p\sim\sum_{k=1}^\infty kB_k(\frac{1}{p})^k,$$
where $B_1=\frac12$, $B_2=\frac16$, $B_3=0$, $\cdots$ are the Bernuolli numbers of the second
kind.
\end{remark}

\begin{proof}
Fix $p$. For 
$$k\ge p, a=\frac12(\sqrt{4p^2+1}-(2p+1))>-\frac12,$$
we have $$a^2+(2p+1)a+p=0$$
and $$a^2+(2p+1)a+p+(2a+1)(k-p)\ge 0,$$
i.e.
$$k^2\le(k+a)(k+1+a).$$
Hence
\begin{align*}
\zeta(2,p)\ge&\sum_{k=p}^{\infty}(\frac{1}{k+a}-\frac{1}{k+1+a})
=\frac{1}{p+a}.
\end{align*}
For
$$k\ge p, b=p(\sqrt{p^2+1}-p)<\frac12,$$
we have
$$b^2+2p^2b-p^2=0$$
and
$$b^2+2p^2b-p^2+(2b-1)(k^2-p^2)\le 0,$$
i.e.
$$k^4\ge(k^2-k+b)(k^2+k+b).$$
Hence
\begin{align*}
\zeta(3,p)<&\frac12\sum_{k=p}^\infty(\frac{1}{(k-1)k+b}-\frac{1}{k(k+1)+b})=\frac{1}{2(p^2-p+b)}.
\end{align*}
Therefore,
\begin{align*}
Q_p=&2p^2\zeta(3,p)-p\zeta(2,p)<\frac{2p^2}{2p(\sqrt{p^2+1}-1)})-\frac{2p}{\sqrt{4p^2+1}-1}\\
=&\frac{\sqrt{p^2+1}+1}{p}-\frac{\sqrt{4p^2+1}+1}{2p}=\frac1{2p}+\frac{2\sqrt{p^2+1}-\sqrt{4p^2+1}}{2p}\\
=&\frac1{2p}+\frac{3}{2p(2\sqrt{p^2+1}+\sqrt{4p^2+1})}<\frac1{2p}+\frac{3}{8p^2}.
\end{align*}
Meanwhile,
\begin{align*}
\zeta(2,p)<&\frac1{p^2}+\sum_{k=p+1}^\infty\frac{1}{(k-\frac12)(k+\frac12)}
=\frac{1}{p^2}+\frac{2}{2p+1};\\
\zeta(3,p)>&
\frac{1}{p^3}+\frac12\sum_{k=p+1}^{\infty}(\frac{1}{k^2-k+\frac12}+\frac{1}{k^2+k+\frac12})
=\frac{1}{p^3}+\frac{1}{2(p^2+p+\frac12)}.
\end{align*}
Hence
\begin{align*}
Q_p=&2p^2\zeta(3,p)-p\zeta(2,p)\\
>&
2p^2(\frac{1}{p^3}+\frac{1}{2(p^2+p+\frac12)})-p(\frac{1}{p^2}+\frac{2}{2p+1})\\
=&\frac1p+\frac{1}{2p+1}-\frac{p+\frac12}{p^2+p+\frac12}>\frac1p-\frac1{2p+1}
.
\end{align*}
\end{proof}

 


\end{document}